\documentclass[11pt,usletter]{article}

\usepackage[english]{babel}
\usepackage{amssymb,amsmath,amsthm} 
\usepackage{url,hyperref}
\usepackage{latexsym}
\usepackage{enumerate}

\usepackage{fullpage}

\newtheorem{lemma}{Lemma}
\newtheorem{cor}{Corollary}
\newtheorem{theorem}{Theorem}

\title{Consecutive positive integers with the same number of divisors}
\author{Vasilii A. Dziubenko \qquad and\qquad Vladimir A. Letsko} 

\begin{document} 
\maketitle 
\thispagestyle{empty} 

\begin{abstract} 
We establish new upper bounds for the length of runs of consecutive positive integers each with exactly $k$ divisors, where $k$ is a given positive integer of some special forms. Also we have found exact values of the maximum possible runs for some fixed values of $k$. In addition, we exhibit the run of 11 consecutive positive integers each with exactly 36 divisors, the run of 13 consecutive positive integers each with exactly 12 divisors, the run of 17 consecutive positive integers each with exactly 24 divisors, and the run of 17 consecutive positive integers each with exactly 48 divisors. 
\end{abstract} 

\section{Function $M(k)$}

Unless explicitly stated otherwise, integers in this paper are assumed to be positive integers.
For a integer $n$, let $\tau(n)$ denote the number of divisors of  integer $n$.
Following \cite{DE,A1}, we say that integers $m$ and $n$ are {\it equidivisible} if $\tau(m)=\tau(n)$.

It is commonly known that 
\begin{equation}
    \tau(n)=\prod^{s}_{i=1}{(\alpha_i+1)},\quad \text{where } n=\prod^s_{i=1}p_i^{\alpha_i} \text{ is the prime factorization of } n.
\label{tau}\end{equation}
Equation \eqref{tau} trivially implies that $n$ is an exact square if and only if $\tau(n)$ is odd. Since consecutive integers cannot all be exact squares, the number of divisors of any consecutive equidivisible integers must be even.
At the same time, for some even $k$, one can find relatively long runs of consecutive integers with exactly $k$ divisors each. For example, 
$\tau(242)=\tau(243)=\tau(244)=\tau(245)=6$. On the other hand, as we will see from Lemma~\ref{lem1} below, the maximal length of such runs is bounded for any $k$, which inspires the following definition.

For an integer $k>0$, we define $M(k)$ as the maximal length of the run of consecutive equidivisible integers with exactly $k$ divisors each. 

It is clear that $M(k)=1$ for all odd $k$. Therefore we will always assume that $k$ is even. It is also clear that $M(2)=2$ and $M(4)=3$. Exact values and upper bounds of $M(k)$ for some other even $k$ are given in \cite{DE,A1,GP}. 

In the present paper, we give several new upper and lower bounds of $M(k)$ for $k$ of some special forms. We also obtain many new exact values of $M(k)$. Finally, we present three longest known runs of 17 consecutive equidivisible integers.

\section{Upper bounds for $M(k)$}\label{sec2}


\begin{lemma}\label{lem1} 
Let $k>1$ be an integer and $p$ be the smallest prime such that $p\nmid k$. Then $M(k) \le 2^p-1$.
\end{lemma}

\begin{proof}
It is easy to see that among any $2^p$ consecutive integers there exists an integer $n$ such that $n \equiv 2^{p-1} \pmod{2^p}$.
Then $p\mid\tau(n)$ and thus $\tau(n)\ne k$, implying that $M(k)<2^p$.
\end{proof}

\begin{cor}
Let $k$ be an integer not divisible by $3$. Then $M(k) \le 7$.
\end{cor}
 
\begin{lemma}\label{lem2} 
Suppose that an integer $k$ is divisible by $2^s$ but not by $2^{s+1}$. Then $M(k) \le 2^{2^s+1}-1$.
\end{lemma}

\begin{proof}
Assume that $M(k) \geq 2^{2^s+1}$ and consider a run of $2^{2^s+1}$ consecutive integers with $k$ divisors each.
Among them, there exist two that are congruent to $2^{2^s-1}$ and $3\cdot 2^{2^s-1}$ modulo $2^{2^s+1}$.
Let $a$ be the smallest of these two integers. Then the other one is $b=a+2^{2^s}$. We have $a=2^{2^s-1}a'$ and $b=2^{2^s-1}(a'+2)$, where $a'$ is odd.

By multiplicativity of $\tau(n)$, we have $k=\tau(a)=\tau(2^{2^{s-1}}a')=2^sk_1$ and $k=\tau(b)=\tau(2^{2^{s-1}}(a'+2))=2^sk'$, where $k'=\tau(a')=\tau(a'+2)$ is odd (since $2^{s+1}\nmid k$). That is, both $a'$ and $a'+2$ must be squares, which is impossible.
\end{proof}

\begin{theorem}\label{th1}
For any odd prime $p$ or $p=9$, $M(6p) \le 5$.
\end{theorem}

\begin{proof}
For $p=3$, it is known that $M(18)=5$ \cite{A1}. For the rest assume that $p>3$. 

It is easy to see that for an integer $n\equiv 6\pmod{8}$, it is not possible to have $\tau(n)=6p$. Therefore, any run of integers with $6p$ divisors must occur within consecutive integers $n_7, n_0, n_1, n_2, n_3, n_4, n_5$, where $n_i\equiv i\pmod8$. 
If the number of divisors of either of $n_0,n_1,n_2,n_3,n_4$ is not $6p$, then the statement holds trivially. 
We therefore assume that $\tau(n_0)=\tau(n_1)=\tau(n_2)=\tau(n_3)=\tau(n_4)=6p$.

Since $n_2\equiv 2\pmod8$ and $\tau(n_2)=6p$, we have $n_2= 2q_2^{p-1}r_2^2$ or $n_2=2r_2^{3p-1}$, where $q_2,r_2$ are distinct odd primes. In particular, $n_2=2x^2$ for some odd integer $x$. Since $n_1=n_2-1=2x^2-1 \not\equiv 0 \pmod 3$, either $n_2$ or $n_3$ is divisible by $3$. 

\begin{enumerate}[(i)]
 \item Suppose that $n_2=2x^2$ is divisible by 3, implying that it is also divisible by 9. If $\tau(n_7)=6p$, then $n_7=3y^2$ for some integer $y$. However, $3y^2 \not\equiv 7 \pmod 8$. Hence $\tau(n_7) \ne 6p$. Similary if $\tau(n_5)=6p$, then $n_5=3z^2$ for some integer $z$. However, $3z^2 \not\equiv 5 \pmod 8$.  Therefore $\tau(n_5)\ne 6p$. 
 
 Thus in this case  $M(6p)\le 5$.

 \item Suppose that $n_3$ is divisible by 3. If $9\mid n_3$, then $n_0=3y^2$ for some integer $y$ and thus $n_3=3y^2+3 \not\equiv 0 \pmod 9$, a contradiction implying that $9\nmid n_3$. 
We have $n_2=2x^2$ and $n_3=3t^2$ for some integers $x,t$. Hence, neither of $n_2,n_3$ is congruent 1 or 4 modulo 5, implying that $5\mid n_0$.

If $p$ is the odd prime, then $n_0=45 \cdot 2^{p-1}$ or $n_0=75 \cdot 2^{p-1}$.  Since the smallest positive integer $n$ coprime to $2\cdot 3\cdot 5$ with $\tau(n)=6p$ divisors is $n=13\cdot 11^2\cdot 7^{p-1}$, it cannot be a neighbor of $n_0$, implying that $n_0$ and $n_1$ can not be in the run together.

If $p=9$, then $n_0 \in \{2^8\cdot 3\cdot5^2, 2^8\cdot 3^2\cdot 5, 2^5\cdot 3^2\cdot 5^2\}$. However, for all these cases $n_1$ has not $6p$ divisors. 

Thus, in case (ii) we showed that $n_0$ can not be in the run.
\end{enumerate} 

\end{proof}

Note, that all known runs of integers each having $6p$ divisors for $p$ described in Theorem~\ref{th1}  start with $n_1$.

Now let $k$ be congruent to 2 or 10 modulo 12.

\begin{lemma}\label{lem3} 
If $k \equiv 2 \pmod{12}$ or $k \equiv 10 \pmod{12}$ then $M(k) \le 5$.
\end{lemma}

\begin{proof}
By Lemma~\ref{lem1}, $M(k) \le 7$ and possible remainders of consecutive integers each with $k$ divisors are: $5,6,7,0,1,2,3$. By Lemma~\ref{lem2}, numbers with remainders 6 and 2 cannot be simultaneously present in the run of integers each with required number of divisors.
\end{proof}

Just as in Theorem~\ref{th1} it can be shown that number congruent to 6 modulo 8 cannot belong to the run.

\begin{lemma}\label{lem4} 
Let $p,q$ be (not necessary different) primes greater than $3$. Then $M(2pq) \le 4$.
\end{lemma}

\begin{proof}
Suppose that there exists the run of 5 consecutive integers with $2pq$ divisors. According to the above remark it must consist of the numbers $n_7,n_0,n_1,n_2,n_3$ where $n_i \equiv i \pmod 8$. 

By the same arguments as in the proof of Theorem~\ref{th1}, one can show that:
\begin{enumerate}
\item If $n_1 \equiv 0 \pmod 3$, then $n_1$ and $n_2$ cannot be equidivisible since $2x^2-1 \not\equiv 0 \pmod 3$.
\item If $n_2 \equiv 0 \pmod 3$, then $n_7$ and $n_2$ cannot be equidivisible since $3x^2+3 \not\equiv 0 \pmod 9$.
\item If $n_3 \equiv 3 \pmod 9$ or $n_3 \equiv 6 \pmod 9$, then $n_0$ and $n_1$ cannot be equidivisible since $n_0=2^{p-1}\cdot 3^{q-1}\cdot 5$ or $n_0=2^{p-1}\cdot 3\cdot 5^{q-1}$.
\item If $n_3 \equiv 0 \pmod 9$, then $n_0$ and $n_3$ cannot be equidivisible since $3x^2+3 \not\equiv 0 \pmod 9$.
\end{enumerate}
\end{proof}

Duentsch and Eggleton~\cite{DE} showed that $M(2p) \le 3$ for each prime $p$ greater than 3. We managed to extend this inequality for some other $k$ cogruent to 2 or 10 modulo 12.

We need the following obvious statement:

\begin{lemma}\label{lem5} 
Let $n,s$ be integers and $s$ be odd, $n>1$. Then $\gcd(n-1,\sum_{i=0}^{s-1} n^{s-1-i})$ and $\gcd(n+1,\sum_{i=0}^{s-1} (-1)^i n^{s-1-i})$ both divide $s$.
\end{lemma}

\begin{theorem}\label{th2}
Let $k=2pq$ for some primes $p, q$  and $d=\gcd(p-1,q-1)$. If $d>4$ then $M(k) \le 3$.
\end{theorem}

 \begin{proof}
 We will keep to the notation introduced in the proof of Lemma~\ref{lem4}.
 
 By Lemma 4, if there exists a run of 4 consecutive integers each with $k$ divisors, then it would contain numbers $n_0$ and $n_2$. Possible prime factorizations of $n_2$ are $2r_2^{pq-1}, 2r_2^{p-1}s_2^{q-1}$. 
 
 First we note that every number in the run (not only $n_2$) cannot contain more than three different prime factors in its prime factorization.
 
 \begin{enumerate}
     \item  4 is a divisor of $d$. Then $n_2=2a^{4c}$ where $c>1$, $a$ is odd.
  \begin{equation}
 n_0=n_2-2=2(a^c-1)(a^c+1)(a^{2c}+1)  \label{21}
 \end{equation}

$\gcd(a^c-1, a^c+1)=2$, $a^{2c}+1$ is divisible by 2 and is not divisible by 4 and the factors of \eqref{21} have no common divisors except 2. Neither $a^c-1$ nor $a^c+1$ can be a power of 2. It follows from the Mih\u{a}ilescu's theorem (the former Catalan's conjecture)~\cite{MT} that $(8,9)$ is the only pair of consecutive integers, each of which is a nontrivial power. The fact that $a^c$ cannot be equal to 9 is verified directly. Thus, the prime factorization of $n_0$ includes the number 2 and at least three other different prime numbers. But this is impossible. Therefore, the numbers $n_0$ and $n_2$ cannot simultaneously belong to the run of consecutive integers having $2pq$ divisors.

\item  $c$ is an odd prime divisor of $d$. Then $n_2=2r_2^{p-1}s_2^{q-1}$ or $n_2=2r_2^{pq-1}$. Anyway $n_2=2a^{2c}$ for some odd $a$. 
\begin{equation}\begin{array}{c}
 n_0=n_2-2=2(a^c-1)(a^c+1)=\\2(a-1)(a^{c-1}+a^{c-2}+\dots+a+1)(a+1)(a^{c-1}-a^{c-2}+\dots-a+1)  \label{22}
\end{array} \end{equation}

Again, none of the numbers $a^c-1$, $a^c+1$ can be a power of 2. Since  $(a^c-1)/2$ and $(a^c+1)/2$ are coprime, there are at least 3 different primes in the prime factorization of $n_0$. Exponent of the prime 2 cannot be less than 4. Thus the only possible representation of $n_0$ is $2^{p-1}r_0^{q-1}s_0$ for some odd primes $r_0$ and $s_0$. 
 
Using this expression in \eqref{22} we obtain the following two cases:
 \begin{equation}
 a-1=2^{p-3},\quad a^{c-1}+a^{c-2}+\dots +a+1=s_0, \quad a^c+1=2r_0^{q-1} \label{p1}
 \end{equation}
\begin{equation}
 a+1=2^{p-3}, \quad a^{c-1}-a^{c-2}+\dots -a+1=s_0, \quad a^c-1=2r_0^{q-1} \label{p2}
 \end{equation}

Consider the equalities \eqref{p1}.  The third of them implies that
 $$a+1=2r_0^{\alpha},\quad  a^{c-1}-a^{c-2}+\dots -a+1=r_0^{\beta}.$$   
From the equalities $a-1=2^{p-3}$ and $a+1=2r_0^{\alpha}$ we get $r_0^{\alpha}-2^{p-4}=1$. It is easely verified that $r_0\not=3$. Hence, by Mih\u{a}ilescu's theorem, $\alpha=1$.
From Lemma 5 we get $r_0=c$. Hence $a+1=2c$ and $a-1=2c-2=2^{p-3}$. Thus $c=2^{p-4}+1$. Since $c$ is odd prime factor of $p-1$ then $p \ge 7$. Therefore $2^{p-4}+1 > p-1$ and $c$ cannot divide $p-1$. Thus, our assumption is refuted so $M(k) \le 3$.

The option \eqref{p2} is considered similarly.
\end{enumerate} 
\end{proof}


\section{Exact values of $M(k)$ for some even $k$}

It is clear that in order to obtain the exact value $M(k)$ for a given $k$ it is sufficient to give an example of a run of consecutive integers having $k$ divisors whose length is equal to the upper bound for $M(k)$. 
 The corresponding upper bounds are given in Corollary 1 and Theorems 1 and 2.
 
 
 Table~\ref{table1} gives the exact values of all known $M(k)$ for even $k$.
 
 Let us consider in more detail the rows of Table~\ref{table1} from the bottom up.
 
Let $k$ be divisible by 4 and is not divisible by 3. The most difficult task is to find 7 consecutive numbers having $k$ divisors where $k$ has large prime factors. That is why Table~\ref{table1} contains, for instance, 2560 but not 92.
 
In the process of searching such a chain we used a certain technique. 
Consider the search for a chain of 7 consecutive integers having $k$ dividers, for the case $k = 76$. 

Applying Chinese remainder theorem one can select the moduli so that each of 7 consecutive integers will be the product of a number having 38 divisors and some other one.  For instance, taking $a=3\cdot 17^{18}$, 
\begin{scriptsize}
$$s_0=283,068,009,891,526,033,048,495,522,741,168,151,673,505,189,503,207,144,294,857,813,203,632,566,487,752,480,287\,,$$ 
\end{scriptsize}
$m=2^{19}\cdot 3^{17}\cdot 5^{18}\cdot 7^{18}\cdot 11^{18}\cdot 13^{18}\cdot 19\cdot 23$, 
$n_j=a(s_0+jm)$
we provide  
\begin{equation}
\begin{array}{l}
n_j=3\cdot 17^{18} q_{0j};\\
n_j+1=2\cdot 5^{18} q_{1j};\\
n_j+2=23\cdot 7^{18} q_{2j};\\
n_j+3=3\cdot 2^{18} q_{3j};\\
n_j+4=19\cdot 11^{18} q_{4j};\\
n_j+5=2\cdot 13^{18} q_{5j};\\
n_j+6=5\cdot 3^{18} q_{6j}.\\
\end{array} 
\label{k76_1}
\end{equation} 
To obtain the required run, it is sufficient to find $j$ such that the numbers $q_{ij}$ are prime for all $i\in \{0,1,2,3,4,5,6\}$.

\begin{table}[!ht]
\begin{center}
\begin{tabular}{|p{8mm}|p{138mm}|}
\hline $M(k)$ & $k$ \\
\hline \hline 2 & 2 \\
\hline 3 & 4, 10, 14, 22, 26, 34, 38, 46, 50, 58, 62, 74, 82, 86, 94, 98, 106, 118, 122, 130, 134, 142, 146, 158, 166, 170, 178, 182, 194, 202, 206, 214, 218, 226, 242, 254, 262, 266, 274, 278, 290, 298, 302, 314, 326, 334, 338, 346, 358, 362, 370, 382, 386, 394, 398, 410, 422, 434, 442, 446, 454, 458, 466, 478, 482, 494, 502, 514, 518, 526, 530, 538, 542, 554, 562, 566, 578, 586, 602, 610, 614, 622, 626, 634, 662, 674, 682, 694, 698, 706, 718, 722, 730, 734, 746, 754, 758, 766, 778, 794, 802, 806, 818, 838, 842, 854, 862, 866, 878, 886, 890, 898, 902, 938, 962, 970, 986, 1010, 1022, 1058, 1066, 1090, 1106, 1118, 1130, 1178, 1258, 1342, 1358, 1370, 1378, 1394, 1406, 1442, 1490, 1526, 1562, 1570, 1586, 1634, 1682, 1730, 1742, 1778, 1802, 1810, 1898, 1922, 1930, 1946, 1970, 2054, 2074, 2114, 2146, 2198, 2222, 2282, 2290, 2294, 2314, 2318, 2330, 2378, 2410, 2482, 2494, 2522, 2534, 2542, 2546, 2570, 2626, 2666, 2678, 2690, 2702, 2738, 2770, 2774, 2786, 2810, 2834, 2882, 2930, 2938, 2954, 3002, 3026, 3034, 3074, 3082, 3122, 3182, 3206, 3298, 3302, 3322, 3362, 3374, 3434, 3562, 3614, 3686, 3698, 3706, 3782, 3794, 3842, 3874, 3878, 3914, 3922, 3926, 3962, 3982, 4082, 4094, 4118, 4142, 4154, 4202, 4234, 4238, 4298, 4346, 4382, 4402, 4418, 4498, 4514, 4526, 4634, 4642, 4658, 4706, 4718, 4826, 4886, 4898, 4958, 5002, 5018, 5066, 5122, 5138, 5162, 5174, 5222, 5246, 5282, 5302, 5338, 5402, 5486, 5522, 5618, 5626, 5738, 5762, 5798, 5822, 5846, 5858, 5882, 5894, 5954, 5962, 5966, 5986, 6014, 6062, 6106, 6146, 6154, 6182, 6194, 6262, 6266, 6278, 6322, 6386, 6466, 6482, 6554, 6562, 6586, 6698, 6758, 6794, 6818, 6878, 6962, 6986, 7046, 7178, 7202, 7282, 7298, 7322, 7334, 7358, 7366, 7442, 7474, 7562, 7622, 7658, 7738, 7786, 7874, 7922, 7946, 7954, 7982, 7994, 8018, 8066, 8078, 8122, 8138, 8174, 8194, 8282, 8342, 8362, 8374, 8414, 8474, 8542, 8642, 8806, 8618, 8662, 8702, 8738, 8762, 8822, 8858, 8906, 8938, 8978, 9074, 9106, 9154, 9158, 9262, 9266, 9362, 9374, 9398, 9434, 9482, 9542, 9554, 9638, 9698, 9718, 9734, 9782, 9854, 10034, 10082, 10106, 10138, 10282, 10286, 10298, 10322, 10498, 10526, 10586, 10634, 10706, 10742, 10754, 10858, 10922, 11174, 11234, 11426, 11458, 11534, 11618, 11666, 11834, 11842, 11894, 11926, 11954, 11966, 12002, 12062, 12238, 12322, 12338, 12382, 12482, 12566, 12578, 12806, 12986, 12994, 12998, 13066, 13082, 13298, 13394, 13502, 13802, 13826, 13862, 13886, 14018, 14162, 14198, 14282, 14342, 14606, 14726, 14842, 14942, 15038, 15226, 15326, 15494, 15562, 15566, 15614, 15662, 15826, 15914, 15982 \\
\hline 5 & 6, 18, 30, 42, 54, 66, 78, 102, 114, 138, 174, 186, 222, 246, 258, 282  \\
\hline 7 & 8, 16, 20, 28, 32, 40, 44, 52, 56 , 64, 68, 76, 80, 88, 100, 104, 112, 128, 140, 160, 176, 196, 200, 220, 224, 256, 280, 320, 352, 400, 448, 500, 512, 560, 640, 800, 896, 1024, 1120 1280, 1792, 2048, 2560, 4096 \\
\hline \end{tabular}
\caption{Exact values of $M(k)$ and corresponding even $k$}
\label{table1}
\end{center}
\end{table}

An empirical calculation of the probability shows that it would most likely take about $10^{12}$ steps to find such $j$. Therefore, we preferred another method of searching the required run of consecutive integers having $76$ divisors.

Let $a=3 \cdot 11^{18}$, 
\begin{scriptsize}
$$s_0=57,367,831,813,930,710,416,942,173,714,046,021,671,610,774,684,867,929,921,010,461,281,856,875,149,028,902,647\,,$$ 
\end{scriptsize}
$m=2^{19}\cdot 3^{17}\cdot 5^{18}\cdot 7^{18}\cdot 13^{18}\cdot 17^{18}$. Then for each $j$ we have
\begin{equation*}
\begin{array}{l}
n_j=3\cdot 11^{18} r_{0j};\\
n_j+1=2\cdot 17^{18} r_{1j};\\
n_j+2=7^{18} r_{2j};\\
n_j+3=3\cdot 2^{18} r_{3j};\\
n_j+4=5^{18} r_{4j};\\
n_j+5=2\cdot 13^{18} r_{5j};\\
n_j+6=3^{18} r_{6j}.\\
\end{array} 
\end{equation*} 
Positive integer $j$ produces the required run if $r_{0j}, r_{1j}, r_{3j}, r_{5j}$ are prime and $r_{2j}, r_{4j}, r_{6j}$ are the products of two different primes. The probability of this event is much greater than the probabilty of simultaneous primality of numbers $q_{ij}$ in \eqref{k76_1}. Unfortunately, the factorization of numbers from the range of interest may take several hours (while primality testing is a very fast procedure). 
However we mostly avoided full factorization in our research.
First of all note that each third multiple of $3^{18}$ is divisible by $3^{19}$, each fifth multiple of $5^{18}$ is divisible by $5^{19}$, and so on. Such a preliminary sieve immediately throws out about 64$\%$ of candidates of $j$.

The main filtation process verifies the primality of the numbers $r_{0j}, r_{1j}, r_{3j}, r_{5j}$ using probabilistic tests. It speeds up the verification and does not lose the possible candidate numbers. The final set is checked using deterministic tests.

At the third stage, the partial factorization of $r_{2j}, r_{4j}, r_{6j}$ is performed. We are only interested in composite numbers that do not have small prime factors and in numbers that have one small prime factor, provided that the remaining factor is also prime.

At the fourth stage we perform the full factorization of $r_{2j}, r_{4j}, r_{6j}$. 

Finally we inspect all numbers of the run by full factorization. 

We have found the run of 7 consecutive integers having 76 divisors by inspection of about $1.2\cdot 10^9$ values of $j$. At the same time, partial factorization was required only for 181 values of $j$. And full factorization was needed only for 4 values of $j$. Note that 3 candidates were rejected just after the factorization of $r_{2j}$. 

The procedure described above allowed us to find 7 consecutive integers having 76 divisors. The first of them is 
\begin{scriptsize}
$$\begin{array}{ll}
2,775,270,598,603,581,528,049,602,020,344,980,538,485,734,694,771,608,751,022,026,551,506,129,981,330,662,646,538,\\924,360,751,256,259,918,212,890,621.
\end{array}$$
\end{scriptsize}
This proves that $M (76) = 7$.

Let $p$ be an odd prime. The method of finding 5 consecutive integers having $6p$ divisors is, in many respects, similar to the above case. But it requires the solution of the system of quadratic congruences.

Arguing as in the proof of Theorem~\ref{th1}, we come to the following system
\begin{equation}
\begin{cases}
2q_1^{p-1}x^2-1 \equiv 0 \pmod{q_0^{p-1}r_0^2},\\
2q_1^{p-1}x^2+1 \equiv 0 \pmod{q_2^{p-1}r_2^2},\\ 
2q_1^{p-1}x^2+2 \equiv 0 \pmod{4q_3^{p-1}},\\
2q_1^{p-1}x^2+3 \equiv 0 \pmod{q_4^{p-1}r_4^2}
\end{cases}
\label{syst} 
\end{equation}
where $q_0,q_1,q_2,q_3,q_4,r_0,r_2,r_4$ some odd primes.

If we let $q_0=7, q_2=3, q_3=5, q_4=11, r_0=17, r_2=19, r_4=29$ then \eqref{syst} has solutions for all odd primes $p$ and $q_1$. Each solution of \eqref{syst} gives a sequence of runs of consecutive integers such that
\begin{equation*}
\begin{array}{l}
n_j=7^{18} 17^2 s_{0j};\\
n_j+1=2\cdot q_1^{18} s_{1j}^2;\\
n_j+2=3^{18} 19^2 s_{2j};\\
n_j+3=4\cdot 5^{18} s_{3j};\\
n_j+4=11^{18} 23^2 3_{4j};\\
\end{array} 
\end{equation*} 
It remains to find $j$ such that $s_{0j},s_{1j},s_{2j},s_{3j},s_{4j}$ will be simultaneously prime. 

The runs of 5 consecutive integers each with $6p$ divisors for $p \in \{5, 7, 11, 13, 17, 23, 31, 67\}$ have been managed with applying system~\eqref{syst}. The runs for $p \in \{19, 29, 37, 41, 43, 47, 53, 59, 61\}$ have been managed with applying of another system similar to system~\eqref{syst}.

The greatest amount of integers for which the exact value of $M (k)$ is proved belong to classes of 2 and 10 modulo 12. 

One can see that there are many $k$ in Table~\ref{table1} such that $k=2pq$ where $\gcd(p-1,q-1)=4$   while Theorem 2 provides $M(k)\le 3$ only for case $\gcd(p-1,q-1)>4$. The fact is that for each of such $k$, represented in Table~\ref{table1}, the inequality $M(k)\le 3 $ is easily verified. 

Suppose, for example, $k=10858$. We will keep the notation of Theorem 2. If there are exist 4 consecutive integers each with $10858=2\cdot61\cdot89$ divisors then $n_0$ and $n_2$ are among these integers. Possible prime factorizations of $n_0$ are $2^{88}r_0^{60}s_0$ or $2^{60}r_0^{88}s_0$ (variants with fewer prime divisors are obviously not suitable). The possibilities for $n_2$ are $2r_2^{60}s_2^{88}$ and $2r_2^{5428}$. Therefore
\begin{equation}
 n_0=n_2-2=2(r_2^{15}s_2^{22}-1)(r_2^{15}s_2^{22}+1)(r_2^{30}s_2^{44}+1)
\label{10858_1} 
\end{equation}
or
\begin{equation}
 n_0=n_2-2=2(r_2^{1357}-1)(r_2^{1357}+1)(r_2^{2714}+1)
\label{10858_2} 
\end{equation}
The only common divisor of factors on the right-hand side of equalities \eqref{10858_1} and \eqref{10858_2} is 2. The fact that none of the factors  can be equal to either $2^{57}$ or $2^{85}$ can be verified directly.

An updatable table containing the numbers opening the runs of $M(k)$ consecutive equidivisible integers for every even $k$ such that exact value of $M(k)$ is proved is located at \cite{MM}.
Note that for $k=2p$ this table represents the smallest numbers opening corresponding runs. These numbers are represented in the sequence A274639 at the OEIS (\cite{OEIS})

\section{Long runs of equidivisible numbers}

The results stated in Section~\ref{sec2} show that, $M(k)$ can be  greater than 7 only provided that $k$ is multiple of 12. 
Firstly let $k=12$.

\begin{lemma}\label{lem6}
$13 \le M(12) \le 15$.
\end{lemma}

\begin{proof}
Combining statements of Lemmas~\ref{lem1} and~\ref{lem2} we get inequality $M(12)\le23$. Since congruence $8x^2\equiv24 \pmod{32}$ has not solutions a run of 23 consecutive integers each with 12 divisors (if it exists) must open with number congruent to 25 modulo 32. Let $n_{25}, n_{26}, \dots, n_{15}$ (where $n_i$ is congruent to $i$ modulo 32) are the numbers of such run. Then $n_0=2^5r_0$ and $n_8=8r_8^2$ for some odd primes $r_0$ and $r_8$. 

One of three consecutive integers must be divisible by 3. If $n_8$ is multiple of 3 then $n_8=72$. But 72 does not belong to required run. Since congruence $8x^2-1\not\equiv0 \pmod3$, $n_7$ cannot be multiple of 3. Finally, if $n_6$ is multiple of 3 then $n_0=96$. But the neighbors of 96 have not 12 divisors.

Thus $n_0$ and $n_8$ cannot at the same time belong to a required run and $M(12)\le15$.

On the other hand number $n=99,949,636,937,406,199,604,777,509,122,843$ starts the run of 13 consecutive integers with 12 divisors each.
\end{proof}

\begin{lemma}\label{lem7}
$11 \le M(36) \le 15$.
\end{lemma}

\begin{proof}
We continue to adhere to the notation adopted earlier.

  As earlier we have $n_8=8x^2$ for some odd $x$ and $n_7$ isn't divisible by 3.  
  
 \begin{enumerate}
 \item
 Let $n_6$ is multiple of 3. Then we have 3 possibilities for $n_0$: $n_0=2^8\cdot 3r_0$; $n_0=2^5\cdot 9r_0$; $n_0=2^5\cdot 3r_0^2$. On other hand $n_0=n_8-8=8(x-1)(x+1)$. Let for instance $8(x-1)(x+1)=2^8\cdot 3r_0$. Then $x-1=48, x+1=2r_0$ or $x+1=48, x-1=2r_0$. Immediate inspection of each of these cases rules that $n_0$ does not belong to required run. The fact that the other two possibilities for $n_0$ are also impossible is proved similarly.   

\item Let $n_8$ be a multiple of 3. Then it is multiple of 9. Hence $n_2=6y^2$ for some odd $y$. Therefore $6y^2+6=72x^2$, which has no integer solutions. 
\end{enumerate} 

Thus $M(36)\le 15$. Taking into account that 
\begin{footnotesize}
$$12,821,655,678,011,960,184,516,598,560,606,241,547,734,025,340,946,441,558,430,971$$
\end{footnotesize}
starts a run of 11 consecutive numbers each with 36 divisors we obtain the required inequality. 
\end{proof}

The longest runs of consecutive equidivisible integers have been found for $k = 24$ and $k = 48$.

\begin{lemma}\label{lem8}
$17 \le M(48) \le 31$.
\end{lemma}

\begin{proof}
$6,611,413,170,876,398,465,463,663,454,441,440,157,066,140$ starts a run of 17 consecutive Integers with exactly 48 divisors. On the other hand, by Lemma~\ref{lem1} $M(48)\le 31$.
\end{proof}

\begin{lemma}\label{lem9}
$17 \le M(24) \le 31$.
\end{lemma}

\begin{proof}
By Lemma~\ref{lem1}, $M(24)\le31$. 
Since $n=768,369,049,267,672,356,024,049,141,254,832,375,543,516$ starts the run of 17 consecutive integers each with 24 divisors $M(24)\ge17$.
\end{proof}

Taking into account that the chains indicated in Lemmas~\ref{lem8} and~\ref{lem9} are the longest known one we will describe them in more detail. The prime factorizations of 17 numbers each with 24 divisors are:
\begin{equation*}
\begin{array}{l}
n=2^2\cdot 44449284079\cdot 105284315902411137115055983\cdot 41047\\
n+1=17^2\cdot 426420806453\cdot 10362291056766174583433\cdot 601697\\
n+2=2\cdot 3\cdot 37^2\cdot 93543833609407396642810949751014411437\\
n+3=7\cdot 29^2\cdot 4175425569259650195904030211773043\cdot 31259\\
n+4=2^5\cdot 5\cdot 4802306557922952225150307132842702347147\\
n+5=3^2\cdot 74040653\cdot 271437693517\cdot 4248023949796631302769\\
n+6=2\cdot 31^2\cdot 71\cdot 5630644789521422491419216640931778631\\
n+7=11^2\cdot 47159139991\cdot 1431258818623\cdot 94080678371991491\\
n+8=2^2\cdot 3\cdot 10159\cdot 6302859937556783443449561482879157853\\
n+9=5^2\cdot 11077607\cdot 34152250833511421371868073517\cdot 81239\\
n+10=2\cdot 7^2\cdot 11020012489\cdot 711478368155899335085950539083\\
n+11=3\cdot 13^2\cdot 12693428625414204133\cdot 119394125256257485217\\
n+12=2^3\cdot 41^2\cdot 57136306459523524392032208600151128461\\
n+13=19^2\cdot 41499463\cdot 4092839\cdot 12531282557382740386748977\\
n+14=2\cdot 3^2\cdot 5\cdot 8537433880751915066933879347275915283817\\
n+15=23^2\cdot 433842413\cdot 280186335801476584936003\cdot 11949101\\
n+16=2^2\cdot 47340127397857455478880553084361\cdot 210037\cdot 19319\\
\end{array}
\end{equation*}

The prime factorizations of 17 numbers each with 48 divisors are:
\begin{equation*}
\begin{array}{l}
n=2^2\cdot 3\cdot 5\cdot 13\cdot 8476170731892818545466235198001846355213\\
n+1=7^2\cdot 103\cdot 750030097811\cdot 78153181\cdot 22347844613545127933\\
n+2=2\cdot 11^2\cdot 127\cdot 332159\cdot 647633324071243610632168052867207\\
n+3=3\cdot 29^2\cdot 301347598303\cdot 141041383013\cdot 61654212602832919\\
n+4=2^5\cdot 5100253\cdot 361893459271\cdot 111936536869609074152759\\
n+5=5\cdot 17^2\cdot 991723\cdot 3663593992141\cdot 1259298621299483593027\\
n+6=2\cdot 3^2\cdot 2539\cdot 22374521223311\cdot 6465548022049506084248693\\
n+7=23^2\cdot 349\cdot 54051131\cdot 14548475587\cdot 45539779574304556031\\
n+8=2^2\cdot 7\cdot 179\cdot 3121\cdot 422658363974931709265507955814180549\\
n+9=3\cdot 37^2\cdot 106082489\cdot 464773609\cdot 32650093930904028778607\\
n+10=2\cdot 5^2\cdot 36779\cdot 102372709\cdot 35118841555131556240420476493\\
n+11=19^2\cdot 53\cdot 83\cdot 637640401\cdot 6529159957873341513711590609\\
n+12=2^3 \cdot 3\cdot 41^2\cdot 163875995708814159861780275987543132983\\
n+13=11\cdot 13^2\cdot 1471\cdot 475229\cdot 5087440011689810884343571563713\\
n+14=2\cdot 31^2\cdot 854593\cdot 2189577128707\cdot 1838320948158846081607\\
n+15=3^5\cdot 5\cdot 7\cdot 777356045958424275774681182180063510531\\
n+16=2^2\cdot 3371\cdot 10567\cdot 620345221\cdot 74798072495675681513256587\\
\end{array}
\end{equation*}

Another run of 17 consecutive integers each with 48 divisors starts with 
$$19,702,712,619,881,487,242,100,642,851,944,119,672,614,940.$$

To search these runs we have used the technique similar to described above:
\begin{itemize}
\item Applying the Chinese remainder theorem, an arithmetic progression was obtained, each term of which and the subsequent 16 integers are not square free.
\item The preliminary filter provided the rejection of runs for which the numbers were divided into higher powers of primes than those required.
\item Using probabilistic primality tests, we eliminated runs for which the remaining factor of at least one of the numbers $n, n+4, n+10, n+14$ was not prime.
\item Applying partial factorization to the remaining numbers of a run, we rejected those runs for which there are integers whose number of divisions cannot be equal to 24 (48).
\item Finally we produced full factorization of all 17 integers for the runs which has not been eliminated earlier.
\end{itemize}

To find the runs of equidivisible integers longer than 31, we must consider $k$ which is multiple of 120. However, with the current development of computer technology, finding such long runs seems unlikely.

\section{Open Problems}

Dickson's conjecture~\cite{Di} is a generalization of Dirichlet's theorem on arithmetic progressions.
If Dickson's conjecture holds, then following statements about $M(k)$ would also hold:
 \begin{enumerate}
 \item {\it $M(2p)=3$ for all prime $p$ greater than $3$.}
 \item {\it $M(2pq)=3$ for all primes $p, q$ such that $\gcd(p-1,q-1)>4$.}
 \item {\it $M(k)\ge 3$ for all $k$ congruent to $\pm 2$ modulo $12$ excluding $k=2$.}
 \item {\it $M(k)=7$ for all $k$ congruent to $\pm 4$ modulo $12$ excluding $k=4$.}
 \item $M(12)=M(36)=15$.
 \item {\it $M(k)=31$ for all $k$ divisible by $24$ and not divisible by $5$.}
 \item {\it $M(k)$ is unbounded}~\cite{DE}. In other words, Erd\"os conjecture~\cite[Problem B18]{GUY} is true. 
  
\vspace{2 mm}Dickson's conjecture is further extended by Schinzel's hypothesis H~\cite{SIER}. Schinzel's conjecture implies:
\item {\it $M(6p)=5$ for all odd prime $p$.}
\item {\it $M(k)\ge 5$ for all $k$ congruent to $6$ modulo $12$.}
 \end{enumerate}
 
 The answers to the following questions are unknown:
 \begin{enumerate}
 \item Are $M(k)=3$ for all $k$ congruent to $\pm 2$ modulo 12 excluding $k=2$? 
 \item Are $M(k)=5$ for all $k$ congruent to 6 modulo 12?
 \item Is $k = 2$ the only value of $k$ for which $M(k)$ is even? 
 \end{enumerate}
Finally, let $D_M$ is the set of possible values of $M(k)$. We know that $1,2,3,5,7$ belong to $D_M$. It is of interest to obtain a complete description of $D_M$.

\section*{Acknowledgements}
We thank Max Alekseyev for careful reading our original manuscript and suggesting a number of corrections.

\medskip
(V. Dziubenko) Department of Control and Applied Mathematics, Moscow Institute of Physics and Technology, Dolgoprudny, Russia

{\it E-mail address:}  \href{mailto:dzvasek257@mail.ru}{dzvasek257@mail.ru}

\smallskip
(V. Letsko) Department of Mathematics, Computer Science and Physics, Volgograd State Socio-Pedagogical University, Volgograd, Russia

{\it E-mail address:}  \href{mailto:val-etc@yandex.ru}{val-etc@yandex.ru}

\end{document}